\documentclass[11pt]{amsart}
\usepackage{mathrsfs}
\usepackage{amssymb}

\numberwithin{equation}{section}
\newtheorem{statement}{}[section]

\newtheorem*{Corollary*}{Corollary}
\newtheorem{Theorem}[statement]{Theorem}
\newtheorem*{Theorem*}{Theorem}
\newtheorem{Lemma}[statement]{Lemma}
\theoremstyle{definition}

\allowdisplaybreaks

\usepackage{array}

\def\HH{\mathscr{H}}

\def\HH(b){\mathcal{H}(b)}

\def\D{\mathbb{D}}
\def\T{\mathbb{T}}

\def\phi{\varphi}

\newcommand{\Span}{\operatorname{span}}

\renewcommand{\ker}{\operatorname{ker}}
\renewcommand{\dim}{\operatorname{dim}}

\newcommand{\ind}{\operatorname{Ind}}
\newcommand{\ran}{\operatorname{ran}}

\newcommand{\beqa}{\begin{eqnarray*}}
\newcommand{\eeqa}{\end{eqnarray*}}

\renewcommand{\leq}{\leqslant}

\renewcommand{\geq}{\geqslant}
\renewcommand{\subset}{\subseteq}

\newcommand{\clb}{\mathcal{B}}

\newcommand{\clh}{\mathcal{H}}
\newcommand{\clk}{\mathcal{K}}

\title[The Cowen-Douglas class and de Branges-Rovnyak spaces]{The Cowen-Douglas class and de Branges-Rovnyak spaces}

\author[Fricain]{Emmanuel Fricain}
 \address{Univ. Lille, CNRS, UMR 8524 - Laboratoire Paul Painlevé, F-59000 Lille, France}
 \email{emmanuel.fricain@univ-lille.fr}

\author[Sarkar]{Jaydeb Sarkar}
\address{Statistics and Mathematics Unit, Indian Statistical Institute, 8th Mile, Mysore Road, Bangalore, 560059, India}
\email{jay@isibang.ac.in, jaydeb@gmail.com}

\thanks{Emmanuel Fricain and Jaydeb Sarkar were supported by the Labex CEMPI (ANR-11-LABX -0007-01). Jaydeb Sarkar is also supported in part by TARE (TAR/2022/000063) by SERB, Department of Science \& Technology (DST), Government of India. }


\begin{document}

\begin{abstract}
We establish a connection between the de Branges-Rovnyak spaces and the Cowen-Douglas class of operators which is associated with complex geometric structures.  We prove that the backward shift operator on a de Branges-Rovnyak space never belongs to the Cowen-Douglas class when the symbol is an extreme point of the closed unit ball of $H^\infty$ (the algebra of bounded analytic functions on the open unit disk). On the contrary, in the non extreme case, it always belongs to the Cowen-Douglas class of rank one. Additionally, we compute the curvature in this case and derive certain exotic results on unitary equivalence and angular derivatives.
\end{abstract}
\maketitle

\tableofcontents

\section{Introduction}\label{sec: intr}

In this paper, we bring together two distinct and classical theories that have so far remained largely unconnected: the Cowen-Douglas class of operators and the de Branges-Rovnyak spaces. Recall that the Cowen-Douglas class of operators is associated with complex geometric structures, specifically hermitian holomorphic vector bundles. Its curvature corresponds to the Chern connection for such bundles, and the problem of complete unitary invariance is examined through the equivalence or equality of curvatures. One important aspect of the Cowen-Douglas class of operators is the assumption of an abundance of eigenvectors in terms of rich point spectrums.

On the other hand, the de Branges-Rovnyak spaces are primarily function-theoretic objects that fundamentally depend on the nature of the symbol which parametrizes the space.  To be more specific, we recall the standard notation in which $H^\infty$ denotes the commutative Banach algebra of all bounded analytic functions on the unit disk $\D = \{z \in \mathbb{C}: |z| < 1\}$. The norm on $H^\infty$ is the standard supremum norm $\|\cdot\|_\infty$ on $\D$. Set
\[
H^\infty_1 = \{b \in  H^\infty: \|b\|_\infty \leq 1\}.
\]
Given a function $b \in H^\infty_1$, the Toeplitz operator $T_b$ is a contraction on the Hardy space $H^2$. Here, $H^2$ denotes the Hilbert space of analytic functions on $\mathbb{D}$ whose power series coefficients form a square-summable sequence (also see \eqref{def: Hardy space}). The \textit{de Branges-Rovnyak space} corresponding to $b \in H^\infty_1$ is the Hilbert space
\[
\HH(b)=(I-T_bT_{\overline{b}})^{\frac{1}{2}} H^2,
\]
equipped with the inner product
\[
\langle (I-T_bT_{\overline{b}})^{\frac{1}{2}} f, (I-T_bT_{\overline{b}})^{\frac{1}{2}} g \rangle_b=\langle f,g\rangle_2,
\]
where $f,g\in H^2\ominus \ker(I-T_bT_{\overline{b}})^{\frac{1}{2}}$ and $\langle\cdot,\cdot\rangle_2$ denotes the scalar product in $H^2$. Note that $T_b^* = T_{\bar{b}}$, and then, since $T_b$ is a contraction, the operator $I-T_bT_{\overline{b}}$ is positive and so the definition above makes sense.

The key point here is that $\mathcal H(b)$ is contractively contained in $H^2$. Following convention, we denote the Toeplitz operator $T_z$ which has the analytic symbol $z$, as the unilateral forward shift operator $S$,
\[
S(f)= T_z(f)=zf,\qquad f\in H^2.
\]
Its adjoint (with respect to the Hilbert space structure of $H^2$) $S^*$ is given by
\[
S^*(f)=\frac{f-f(0)}{z},\qquad f\in H^2.
\]
A crucial fact is that $\HH(b)$ remains invariant under the backward shift operator $S^*$, and we have the bounded linear operator $X_b: \HH(b) \rightarrow \HH(b)$ defined by
\[
X_b = S^*|_{\HH(b)}.
\]
A central role in the theory of $\HH(b)$ spaces is played by this operator. In particular, it is known that $X_b$ is a contraction on $\HH(b)$ \cite[Theorem 18.7]{FM2}. It should also be noted that a de Branges-Rovnyak space $\mathcal H(b)$ is not uniquely determined by the symbol $b$, that is two different $b$'s could define the same $\mathcal H(b)$ (with a different norm). There exists a characterization of when two de Branges-Rovnyak spaces coincide (with equivalent norms). See \cite{FM1,FM2,Sarason-book} for a comprehensive account of these spaces and the reference therein.

Next, we shift our attention to the other key object of interest in this paper. In \cite{Cowen-Douglas}, Cowen and Douglas initiated a systematic study of a class of bounded linear operators on a complex separable Hilbert space $\mathcal H$ that possess an open set $\Omega$ in $\mathbb{C}$ of eigenvalues with constant and finite multiplicity. More specifically, for a bounded linear operator $T:\mathcal H\longrightarrow\mathcal H$ (in short, $T \in \mathcal B(\mathcal H)$) and a natural number $n$, we say that $T \in B_n(\Omega)$ if the following conditions are satisfied:
\begin{enumerate}
\item $\dim(\ker(T-\omega I))=n$ for all $\omega\in\Omega$.
\item $\overline{\Span}\{\ker(T-\omega I):\omega \in \Omega\} = \mathcal H$.
\item $\ran(T-\omega I)=\mathcal H$ for all $\omega\in\Omega$.
\end{enumerate}

We also say that $T$ is a \textit{Cowen-Douglas class operator of rank $n$} on $\Omega$. Cowen and Douglas proved that for a given $T \in B_n(\Omega)$, there exists a hermitian holomorphic vector bundle $E_T$ of rank $n$ over $\Omega$ (see Section \ref{sec: invar}). Of course, our main interest is in the case $T=X_b$, $b \in H^\infty$ satisfying $\|b\|_\infty \leq 1$. More specifically, we seek to determine whether such an operator $X_b$ belongs to the Cowen-Douglas class. The following theorem, a central result of this paper, provides a precise answer to this question. As it is often the case in the theory of de Branges-Rovnyak spaces, the answer depends whether $b$ is an extreme point of $H^\infty_1$ or not. Recall that an extreme point of $H^\infty_1$ is a function in $H^\infty_1$ that does not lie in any open line segment connecting two distinct points in $H^\infty_1$ \cite[page 214]{FM1}. It is well-known that a function $b \in H^\infty_1$ is an extreme point if and only if \cite[Theorem 6.7]{FM1}
\[
\int_{\T} \log (1 - |b(z)|^2) \, dm = - \infty.
\]
We recall that a function $b \in H^\infty$ admits boundary values (in terms of radial limits) almost everywhere with respect to the standard Lebesgue measure $m$ on $\T (= \partial \D)$. This justifies the integrand function in the above integration. We can now state our first main result.

\begin{Theorem}\label{thm:CD-class}
Let $b \in H^\infty_1$. Then the following statements hold :
\begin{enumerate}
\item If $b$ is an extreme point, then $X_b \notin B_n(\Omega)$ for any natural number $n$ and open set $\Omega$ in $\mathbb{C}$.
\item If $b$ is a non-extreme point, then $X_b \in B_1(\mathbb D)$.
\end{enumerate}
\end{Theorem}

This result has several significant implications. Here, we highlight two key points. First, it provides a wealth of new and exotic examples of Cowen-Douglas class operators of rank one. We now know that the backward shift operators on de Branges-Rovnyak spaces corresponding to non-extreme points indeed belong to the Cowen-Douglas class $B_1(\D)$. This, in turn, creates a natural link between these two concepts.

Second, Theorem~\ref{thm:CD-class} will will give us  a method to determine whether two different operators $X_{b_1}$ and $X_{b_2}$, acting respectively on two different de Branges-Rovnyak spaces $\mathcal H(b_1)$ and $\mathcal H(b_2)$, are unitary equivalent. Recall that $T_1 \in \clb(\clh_1)$ and $T_2 \in \clb(\clh_2)$ are \textit{unitarily equivalent} if there exists a unitary operator $U: \clh_1 \rightarrow \clh_2$ such that $U T_1 = T_2 U$. We write this simply as
\[
T_1 \cong T_2.
\]
Given $T \in\clb(\clh)$ which belongs to $B_1(\D)$, we continue our discussion of the hermitian holomorphic vector bundle $E_T$ over $\D$ (as Theorem \ref{thm:CD-class} allows us to restrict our attention to the rank one case). In this case, the line bundle $E_T$ yields the \textit{curvature} $\clk_T$ of $T$, where (see \cite[Theorem 1.17]{Cowen-Douglas})
\begin{equation}\label{eq:def-curvature}
\clk_T(\omega)=-\frac{\partial^2}{\partial\omega\partial\overline{\omega}}\log\|\gamma_{T, \omega}\|^2 \qquad (\omega \in \D),
\end{equation}
and $\gamma_{T, \cdot}: \D \longrightarrow \mathcal H$ is any nonzero holomorphic cross-section of $E_T$. The important fact is that this curvature is a complete unitary invariant for operators in $B_1(\D)$. See Section~\ref{sec: invar} for details.

Given a non-extreme point $b \in H^\infty_1$, we consider the \textit{pythagorean pair} $(a,b)$, where $a \in H^\infty_1$ is the unique outer function such that $a(0)>0$ and
\[
|a|^2 + |b|^2 = 1,
\]
a.e. on $\T$. We set the ratio
\[
\Phi_b := \frac{b}{a}.
\]
Let $b_1, b_2 \in H^\infty_1$ be two non-extreme points. Assume that $(a_1,b_1)$ and $(a_2,b_2)$ are the corresponding pythagorean pairs. In Theorem \ref{Thm:Xb-unitary-equivalence}, we prove that
\[
X_{b_1} \cong X_{b_2},
\]
if and only if
\[
\frac{|\Phi_{b_1}'(\omega)|}{1+|\Phi_{b_1}(\omega)|^2} = \frac{|\Phi_{b_2}'(\omega)|}{1 + |\Phi_{b_2}(\omega)|^2} \qquad (\omega\in\mathbb D).
\]

In the context of non-extreme case, there is another operator that interests us. If $b \in H^\infty_1$ is a non-extreme point, then $S \HH(b) \subseteq \HH(b)$, and consequently,  the restriction operator
\[
S_b = S|_{\HH(b)}
\]
is a bounded operator on $\HH(b)$, that is $S_b \in \clb(\HH(b))$. In Theorem \ref{Thm:Sbstar-CD}, we prove that
\[
S_b^*\in B_1(\mathbb D).
\]
Using one more time the fact that the curvature defined in \eqref{eq:def-curvature} is a complete unitary invariant, we get, in Theorem \ref{Thm-Sb-unitary-equivalent}, that for non extreme points $b_1,b_2 \in H^\infty_1$,
\[
S_{b_1} \cong S_{b_2}
\]
if and only if
\[
\frac{|b_1'(\omega)|}{1-|b_1(\omega)|^2}=\frac{|b_2'(\omega)|}{1-|b_2(\omega)|^2} \qquad (\omega\in\mathbb D).
\]

We establish an additional connection between classical Carath\'eodory angular derivatives, de Branges Rovnyak spaces, and the Cowen-Douglas class of operators as follows (see Theorem \ref{cor:ADC}): Let $b_1,b_2 \in H^\infty_1$ be two rational (not inner) functions in $H^\infty_1$. Assume that
\[
S_{b_1} \cong S_{b_2}.
\]
Then $b_1$ and $b_2$ must have the same Carath\'eodory angular derivative points.

The rest of the paper is organized as follows. Section \ref{sec: dB and CD} provides a complete overview of de Branges-Rovnyak spaces and highlights that operators corresponding to non-extreme points belong only to $B_1(\D)$. Section \ref{sec: invar} computes the curvatures of operators associated with de Branges-Rovnyak spaces and presents a clear-cut expression for complete unitary invariants. Section \ref{sec: derivatives} connects our theory with yet another classical notion, namely Carath\'eodory angular derivatives. The final section, Section \ref{sec: examples}, provides a concrete example to illustrate some of the main results of this paper.

\section{The de Branges-Rovnyak spaces in $B_1(\D)$}\label{sec: dB and CD}

Generally, the theory of de Branges-Rovnyak spaces bifurcates into two directions depending on whether the symbol $b \in H^\infty_1$ is an extreme point or not. The purpose of this section is to prove the central result (Theorem \ref{thm:CD-class}) of this paper, which once more explicitly establishes this dichotomy.

We begin by recalling some fundamental and well-known concepts concerning de Branges-Rovnyak spaces. First, the Hardy space $H^2$ is the Hilbert space of holomorphic functions $f$ on $\D$ such that
\begin{equation}\label{def: Hardy space}
\|f\|_2 := \sup_{r \in (0,1)}\left(\frac{1}{2 \pi}  \int_{0}^{2 \pi} |f(r e^{i \theta})|^2\,d\theta\right)^{\frac{1}{2}} < \infty.
\end{equation}
As noted in the introduction, the de Branges-Rovnyak space $\mathcal H(b)$ associated with $b \in H^\infty_1$ is given by the range space
\[
\HH(b)=(I-T_bT_{\overline{b}})^{\frac{1}{2}}H^2.
\]
This space is an algebraic complement of the range space $bH^2$ (a Hilbert space with respect to a similar range-based inner product). It is well known that $\HH(b)$ is a reproducing kernel Hilbert space corresponding to the kernel function $k^b: \D \times \D \rightarrow \mathbb{C}$ defined by
\[
k^b(z,\omega) =k_\omega^b(z)= \frac{1-\overline{b(\omega)}b(z)}{1-\overline{\omega}z},
\]
for all $z, \omega \in \D$. In particular,  for every function $f\in\HH(b)$ and $\omega\in\mathbb D$, we have the \textit{reproducing property} as
\[
f(\omega)=\langle f,k_\omega^b\rangle_b,
\]
and $\{k_{\omega}^b:\omega\in\D\}$ is a total set in $\HH(b)$. It is worth pointing out that the kernel function $k^b_\omega$ is a scaling of the Szeg\"{o} kernel $k_\omega$ of the disc $\D$, where $k_\omega(z)=(1-\overline{\omega}z)^{-1}$ for all $z, \omega \in \D$. More specifically, given that $T_{\bar{b}} k_{\omega} = \overline{b(\omega)} k_{\omega}$, we have
\[
k_\omega^b=(I-T_bT_{\overline{b}})k_\omega,
\]
for all $\omega \in \D$. Also, recall that
\[
X_b f = S^*f \qquad (f \in \HH(b),
\]
defines a contraction on $\HH(b)$ \cite[Theorem 18.7]{FM2}, where $S^*$ is the backward shift operator on $H^2$ given by
\[
S^*f = T_{\overline{z}}f = (f-f(0))/z.
\]
In view of the operator $X_b$, we can define a more general operator $Q_w$ for each $w \in \D$ as follows:
\[
Q_wf(z)=\frac{f(z)-f(\omega)}{z-\omega},
\]
for all $f\in\HH(b)$ and $z \in \D$. Observe that, since $X_b$ is a contraction, the operator $I-\omega X_b$ is invertible for every $\omega\in\D$, and $Q_wf=(I-\omega X_b)^{-1}X_b$, whence $Q_\omega f\in\HH(b)$ for all $f \in \HH(b)$. Moreover, note that $Q_0 = X_b$. \\

Returning to $X_b$, we observe that $S^*b\in\HH(b)$, and for every $f\in\HH(b)$, we have the well-known identity \cite[Theorem 18.22]{FM2}
\begin{equation}\label{eq:adjoint-Xb}
X_b^*f = S f - \langle f,S^*f\rangle_b b.
\end{equation}
When $b \in H^\infty_1$ is a non-extreme point, it follows that $b\in\HH(b)$, and then, the above identity implies that \cite[Corollary 23.9 and Theorem 24.1]{FM2}
\[
S \HH(b) \subseteq \HH(b).
\]
In this case, we write the restriction operator
\[
S_b = S|_{\HH(b)}.
\]
Then \eqref{eq:adjoint-Xb} can be written as
\begin{equation}\label{eq:adjoint-Xb-nonextreme}
X_b^*=S_b-b\otimes S^*b.
\end{equation}
In view of the characterizations of extreme points introduced in Section \ref{sec: intr}, we also recall that a function $b \in H^\infty_1$ is non-extreme if and only if
\[
\int_{\T} \log (1 - |b(z)|^2) \, dm > - \infty.
\]
The following lemma will be the key point to studying the membership of $X_b$ to the Cowen-Douglas class.

\begin{Lemma}\label{lem-Sb-Fredholm}
Let $b$ be a non-extreme point of $H^\infty_1$ and let $\omega\in\mathbb D$. Then $S_b-\omega I$ is a Fredholm operator with
\[
\ind (S_b - \omega I) = -1.
\]
\end{Lemma}
\begin{proof}
By the definition of $S_b$, it is clear that $S_b f = f$ for $f \in \HH(b)$ implies that $f = 0$. This yields $\ker(S_b-\omega I)=\{0\}$ (that is, $\sigma_p(S_b)=\emptyset$). Let us prove now that
\[
\ran(S_b-\omega I)=(\mathbb Ck_\omega^b)^\perp.
\]
First, let $f\in\ran(S_b-\omega I)$, that is $f=(z-\omega)g$ for some $g\in\HH(b)$. Then the reproducing property implies
\[
\langle f,k_\omega^b\rangle_b=f(\omega)=0,
\]
which proves that $f\in (\mathbb Ck_\omega^b)^\perp$. Conversely, let $f\perp k_\omega^b$ for some $f\in\HH(b)$. Using the reproducing property, we have $f(\omega)=0$. We know, in general, that $Q_\omega f \in \HH(b)$, where, in this particular situation, we have
\[
(Q_\omega f)(z) = \frac{f(z)-f(\omega)}{z-\omega}=\frac{f(z)}{z-\omega},
\]
for all $z \in \D$. This implies
\[
f=(S_b-\omega I)Q_\omega f.
\]
Thus we deduce that $(\mathbb Ck_\omega^b)^\perp\subset \ran(S_b-\omega I)$, proving the reversed inclusion. In particular, $\ran(S_b-\omega I)=(\mathbb Ck_\omega^b)^\perp$ proves that $\ran(S_b-\omega I)$ is closed and
\begin{equation}\label{eq:noyau-S-b-adjoint-omega-I}
\ker(S_b^*-\overline{\omega}I) = (\ran(S_b-\omega I))^\perp = \mathbb C k_\omega^b.
\end{equation}
We thus conclude that $S_b-\omega I$ is Fredholm, with index equals to
\[
\ind(S_b-\omega I) = \dim(\ker(S_b-\omega I))-\dim(\ker(S_b^*-\overline{\omega}I)) = -1.
\]
This concludes the proof of the lemma. 
\end{proof}

This result may be of independent interest. With this in place, we are now ready to prove the central result of this paper:

\bigskip

\noindent \textit{Proof of Theorem \ref{thm:CD-class}:} To prove (i), we assume that $b \in H^\infty_1$ is an extreme point. It follows from \cite[Theorem 26.1]{FM2} that the point spectrum of $X_b$ is given by
\[
\sigma_p(X_b)=\{\overline{\omega}:\omega\in\mathbb D,\, b(\omega)=0\}.
\]
In particular, if $X_b\in B_n(\Omega)$ for some $n\geq 1$ and some nonempty open subset $\Omega$ of $\mathbb D$, then
\[
b|_{\Omega} \equiv 0.
\]
By continuation principle, we thus deduce that $b\equiv 0$, which contradicts the fact that the function
\[
\log(1-|b|^2) \notin L^1(\mathbb T).
\]
For (ii), assume that $b \in H^\infty_1$ is not an extreme point. In this case, we recall from \cite[Theorem 24.13]{FM2} that
\[
\sigma_p(X_b)=\mathbb D,
\]
and
\begin{equation}\label{eq:eigenspaces-Xb}
\ker(X_b-\overline{\omega} I)=\mathbb C k_\omega,
\end{equation}
for all $\omega\in\mathbb D$. In particular, $\dim(\ker(X_b-\overline{\omega} I))=1$ for every $\omega\in\mathbb D$. Moreover, we recall from \cite[Corollary 23.26]{FM2} that
\[
\overline{\text{span}}\{ \ker(X_b-\overline{\omega} I):\omega\in \D\} = \overline{\text{span}}\{ k_\omega:\omega\in\mathbb D\} = \HH(b).
\]
Thus, to prove that $X_b \in B_1(\D)$, it only remains to verify that
\[
\ran(X_b-\omega I)=\HH(b),
\]
for all $\omega\in\mathbb D$. Equivalently,
\[
\ran(X_b^*-\overline{\omega} I)\mbox{ is closed},
\]
and
\[
\ker(X_b^*-\overline{\omega} I)=\{0\},
\]
for all $\omega\in\mathbb D$. Fix $\omega\in\mathbb D$. According to \cite[Theorem 24.14]{FM2},
\[
\sigma_p(X_b^*)=\emptyset,
\]
whence $\ker(X_b^*-\overline{\omega} I)=\{0\}$ is satisfied. On the other hand, note that \eqref{eq:adjoint-Xb-nonextreme} implies that
\[
X_b^*-\overline{\omega}I=S_b-\overline{\omega}I-b\otimes S^*b.
\]
According to Lemma~\ref{lem-Sb-Fredholm}, $S_b-\overline{\omega}I$ is a Fredholm operator. Thus $X_b^*-\overline{\omega}I$ is also a Fredholm operator. In particular, $\ran(X_b^*-\overline{\omega}I)$ is closed, which proves that $\ran(X_b-\omega I)=\HH(b)$. Finally, we can conclude that $X_b\in B_1(\mathbb D)$.
\qed

\bigskip

In particular, this result exhibits yet another property of extreme and non-extreme points of $H^\infty_1$. We recall once again that for a non-extreme $b \in H^\infty_1$, we have $S \HH(b) \subseteq \HH(b)$, and we define the restriction operator $S_b = S|_{\HH(b)}$. We claim that $S_b^*$ is in $B_1(\mathbb D)$:

\begin{Theorem}\label{Thm:Sbstar-CD}
Let $b$ be a non-extreme point of $H^\infty_1$. Then $S_b^*\in B_1(\mathbb D)$.
\end{Theorem}

\begin{proof}
Recall that \eqref{eq:noyau-S-b-adjoint-omega-I} says that, for every $\omega\in\mathbb D$, we have
\[
\ker(S_b^*-\overline{\omega}I)=\mathbb C k_\omega^b.
\]
Hence
\[
\dim(\ker(S_b^*-\overline{\omega}I))=1,
\]
and we also trivially get that
\[
\overline{\text{span}}\{\ker(S_b^*-\overline{\omega}I):\omega\in\mathbb D \} = \HH(b).
\]
On the other hand, according to Lemma \ref{lem-Sb-Fredholm}, $S_b-\omega I$ and thus $S_b^*-\overline{\omega}I$ has a closed range, which finally proves that $S_b^*\in B_1(\mathbb D)$.
\end{proof}

A large and important class of extreme points of $H^\infty_1$ are the inner functions. When $b=\Theta$ is an inner function, then $\HH(b)$ coincides with $K_\Theta=(\Theta H^2)^\perp$, a prototype of the family of closed backward shift invariant subspaces of $H^2$. In this case, $X_b$ becomes the adjoint of the model operator $S_\Theta$, where
\[
S_\Theta = P_{K_\Theta} T_z|_{K_\Theta},
\]
and $P_{K_\Theta}$ is the orthogonal projection of $H^2$ onto $K_\Theta$. Theorem \ref{thm:CD-class}, in particular, proves that model operators' adjoints are not in $B_n(\D)$ for any $n \in \mathbb{N}$. In the context of classical models and de Branges–Rovnyak spaces with non-extreme symbols, we refer the reader to \cite{JT}.

\section{Curvatures as unitary invariants}\label{sec: invar}

Operators in $B_n(\Omega)$ are naturally associated with curvature (in view of the Chern connection). In our case, we already know that $X_b$ belongs to the Cowen-Douglas class if and only if $b$ is a non-extreme point, and in that case, $X_b$ is in the Cowen-Douglas class of rank one (that is, $X_b \in B_1(\D)$). However, for the rank-one case, curvature serves as a complete unitary invariant (see Theorem \ref{Thm:cowen-douglas} below). This motivates our study in this section, where we compute the curvature of the line bundle $E_{X_b}$ corresponding to $X_b \in B_1(\D)$. Using Pythagorean pairs corresponding to non-extreme points, we provide computable unitary invariants for $X_b$.

To construct the curvature, we first isolate the hermitian holomorphic vector bundle corresponding to a given $T \in\clb(\clh)$ which belongs to $B_1(\D)$. Given $\omega \in \D$, we define
\[
E_T(\omega) = \{\omega\} \times \ker(T-\omega I).
\]
Then $\omega \mapsto E_T(\omega)$ defines a rank one hermitian holomorphic vector bundle over $\D$ as (see \cite{Cowen-Douglas})
\begin{align*}
E_T :=  & \bigcup_{\omega \in \D} E_T(\omega)
\\
\bigg\downarrow &
\\
\D
\end{align*}
In this case, there exists a holomorphic $\mathcal H$-valued function $\gamma_{T, \cdot}: \D \longrightarrow \mathcal H$ such that for every $\omega\in U$, we have the spanning property
\[
\mathbb{C} \gamma_{T, \omega} = \ker(T-\omega I).
\]
The \textit{curvature} $\clk_T$ of $T$ is then given by
\[
\clk_T(\omega)=-\frac{\partial^2}{\partial\omega\partial\overline{\omega}}\log\|\gamma_{T, w}\|^2 \qquad (\omega \in \D).
\]
The important fact is that this curvature is a complete unitary invariant for operators in $B_1(\D)$ (see Cowen and Douglas \cite{Cowen-Douglas}):

\begin{Theorem}\label{Thm:cowen-douglas} Let $\Omega$ be a domain in $\mathbb{C}$, $\clh_j$ be a Hilbert space, and let $T_j \in \clb(\clh_j)$, $j =1,2$. Assume that $T_1,T_2\in B_1(\Omega)$. Then $T_1 \cong T_2$ if and only if
\[
\clk_{T_1} = \clk_{T_2} \text{ on } \Omega.
\]
\end{Theorem}

Our primary interest lies in the case $T = X_b$. By Theorem \ref{thm:CD-class}, we know that when $b$ is a non-extreme point in $H^\infty_1$, then $X_b\in B_1(\mathbb D)$. Recall from \eqref{eq:eigenspaces-Xb} that $\ker(X_b-\omega I)=\mathbb C k_{\overline{\omega}}$, and hence
\begin{equation}\label{eq:vector-bundle-Xb}
\gamma_{X_b, \omega} = k_{\overline{\omega}} \in \HH(b),
\end{equation}
for all $\omega \in \D$. To get a criterion for unitarily equivalence between two $X_{b}$ operators, we will need the following simple technical lemma.

\begin{Lemma}\label{lem:calcul}
Let $\varphi:\mathbb D\longrightarrow\mathbb C$ be a holomorphic function. Then, for every $\omega\in\mathbb D$, we have
\[
\frac{\partial^2}{\partial\omega\partial\overline{\omega}}\log(1+|\varphi(\omega)|^2)=\frac{|\varphi'(\omega)|^2}{(1+|\varphi(\omega)|^2)^2}.
\]
Moreover, if $\varphi(\mathbb D)\subset\mathbb D$, we also have, for every $\omega\in\mathbb D$,
\[
\frac{\partial^2}{\partial\omega\partial\overline{\omega}}\log(1-|\varphi(\omega)|^2)=-\frac{|\varphi'(\omega)|^2}{(1-|\varphi(\omega)|^2)^2}.
\]
\end{Lemma}
\begin{proof}
Since $\varphi$ is analytic, we have
\[
\frac{\partial}{\partial\overline{\omega}}\log(1+|\varphi(\omega)|^2)=\frac{\partial}{\partial\overline{\omega}}\log(1+\varphi(\omega)\overline{\varphi(\omega)})=\frac{\overline{\varphi'(\omega)}\varphi(\omega)}{1+\overline{\varphi(\omega)}\varphi(\omega)}.
\]
Thus
\[
\begin{split}
\frac{\partial^2}{\partial\omega\partial\overline{\omega}}\log(1+|\varphi(\omega)|^2)=&\frac{\partial}{\partial\omega}\left(\frac{\overline{\varphi'(\omega)}\varphi(\omega)}{1+\overline{\varphi(\omega)}\varphi(\omega)}\right)\\
=&\frac{\overline{\varphi'(\omega)}\varphi'(\omega)(1+\overline{\varphi(\omega)}\varphi(\omega))-\overline{\varphi'(\omega)}\varphi(\omega)\overline{\varphi(\omega)}\varphi'(\omega)}{(1+\overline{\varphi(\omega)}\varphi(\omega))^2}\\
=&\frac{|\varphi'(\omega)|^2+|\varphi'(\omega)|^2|\varphi(\omega)|^2-|\varphi'(\omega)|^2|\varphi(\omega)|^2}{(1+|\varphi(\omega)|^2)^2}\\
=&\frac{|\varphi'(\omega)|^2}{(1+|\varphi(\omega)|^2)^2}.
\end{split}
\]
One can prove the second formula similarly.
\end{proof}

Recall that when $b \in H^\infty_1$ is not an extreme point, then there exists a unique outer function $a$ in $H^\infty$ such that
\[
a(0)>0,
\]
and
\[
|a|^2+|b|^2=1 \text{ a.e. on } \mathbb T.
\]
The pair $(a,b)$ is called a \textit{pythagorean pair}. Given a pythagorean pair $(a,b)$, if $b$ is not an extreme point of $H^\infty_1$, then, for every $\omega\in\mathbb D$, we have (see \cite[Theorem 23.23]{FM2})
\[
k_\omega\in \HH(b),
\]
and \cite[Corollary 23.25]{FM2}
\begin{equation}\label{eq:norme-noyau-Cauchy}
\|k_\omega\|_b^2=\left(1+\frac{|b(\omega)|^2}{|a(\omega)|^2}\right)\cdot \frac{1}{1-|\omega|^2},\qquad \omega\in\mathbb D.
\end{equation}
We are now ready to present the complete unitary invariants for the model operators acting on de Branges-Rovnyak spaces through the lens of the Cowen-Douglas class.

\begin{Theorem}\label{Thm:Xb-unitary-equivalence}
Let $b_1, b_2 \in H^\infty_1$ be two non-extreme points. Assume that $(a_1,b_1)$ and $(a_2,b_2)$ are two pythagorean pairs, and let
\[
\Phi_j = \frac{b_j}{a_j},
\]
for $j=1,2$. Then
\[
X_{b_1} \cong X_{b_2},
\]
if and only if for every $\omega\in\mathbb D$, we have
\[
\frac{|\Phi_1'(\omega)|}{1+|\Phi_1(\omega)|^2}=\frac{|\Phi_2'(\omega)|}{1+|\Phi_2(\omega)|^2},
\]
\end{Theorem}
\begin{proof}
As already pointed out, Theorem \ref{thm:CD-class} implies that $X_{b_1},X_{b_2}\in B_1(\mathbb D)$. Moreover, by \eqref{eq:eigenspaces-Xb}, for every $\omega\in\mathbb D$, we have
\[
\ker(X_{b_1}-\omega I)=\ker(X_{b_2}- \omega I)=\mathbb C k_{\overline{\omega}}.
\]
Now it follows from Theorem~\ref{Thm:cowen-douglas} that $X_{b_1} \cong X_{b_2}$ if and only if for every $\omega\in\mathbb D$, we have
\[
\clk_{X_{b_1}}(\omega) = \clk_{X_{b_2}}(\omega),
\]
where, according to \eqref{eq:vector-bundle-Xb}, we have
\[
\clk_{X_{b_i}}(\omega)=-\frac{\partial^2}{\partial\omega\partial\overline{\omega}}\log\|k_{\overline{\omega}}\|_{b_i}^2=-\frac{\partial^2}{\partial\omega\partial\overline{\omega}}\log\|k_\omega\|_{b_i}^2,
\]
for $i=1,2$. But according to \eqref{eq:norme-noyau-Cauchy}, we have
\[
\|k_\omega\|_{b_i}^2=(1+|\Phi_i(\omega)|^2)\cdot \frac{1}{1+|\omega|^2},
\]
whence
\[
\log\|k_\omega\|_{b_i}^2=\log(1+|\Phi(\omega)|^2)-\log(1-|\omega|^2),
\]
for all $\omega\in\mathbb D$. It follows from Lemma \ref{lem:calcul}, for $i=1,2$, and for all $\omega\in\mathbb D$, that
\[
\clk_{X_{b_i}}(\omega)=-\frac{|\Phi_i'(\omega)|^2}{(1+|\Phi_i(\omega)|^2)^2}+\frac{1}{(1-|\omega|^2)^2}.
\]
Therefore $X_{b_1} \cong X_{b_2}$ if and only if
\[
-\frac{|\Phi_1'(\omega)|^2}{(1+|\Phi_1(\omega)|^2)^2}+\frac{1}{(1-|\omega|^2)^2}=-\frac{|\Phi_2'(\omega)|^2}{(1+|\Phi_2(\omega)|^2)^2}+\frac{1}{(1-|\omega|^2)^2},
\]
or equivalently,
\[
\frac{|\Phi_1'(\omega)|^2}{(1+|\Phi_1(\omega)|^2)^2} = \frac{|\Phi_2'(\omega)|^2}{(1+|\Phi_2(\omega)|^2)^2},
\]
for all $\omega \in \D$. This completes the proof of the theorem.
\end{proof}

Next, we turn to $S_b$ for a non-extreme point of $b \in H^\infty_1$. By Theorem \ref{Thm:Sbstar-CD}, we know that $S_b^*\in B_1(\mathbb D)$. This raises the question of the unitary equivalence of a pair of such operators. The following is our answer:

\begin{Theorem}\label{Thm-Sb-unitary-equivalent}
Let $b_1,b_2 \in H^\infty_1$ be two non-extreme points. Then
\[
S_{b_1} \cong S_{b_2},
\]
if and only if for every $\omega\in\mathbb D$, we have
\begin{equation}\label{eq:ue-Sb}
\frac{|b_1'(\omega)|}{1-|b_1(\omega)|^2}=\frac{|b_2'(\omega)|}{1-|b_2(\omega)|^2}.
\end{equation}
\end{Theorem}

\begin{proof}
We prove the equivalent version that $S^*_{b_1} \cong S^*_{b_2}$. According to Theorem~\ref{Thm:Sbstar-CD}, $S^*_{b_1}$ and $S^*_{b_2}$ are in $B_1(\mathbb D)$, and we can apply Theorem~\ref{Thm:cowen-douglas} which implies that $S^*_{b_1} \cong S^*_{b_2}$ is equivalent to
\[
\clk_{S_{b_1}^*} = \clk_{S_{b_2}^*} \text{ on } \mathbb D.
\]
It follows from \eqref{eq:noyau-S-b-adjoint-omega-I} that, for $i=1,2$, we have
\[
\ker(S_{b_i}^*-\omega I)=\mathbb C k_{\overline{\omega}}^{b_i},
\]
whence
\[
\clk_{S_{b_i}^*}(\omega)=-\frac{\partial^2}{\partial\omega\partial\overline{\omega}}\log\|k_{\overline{\omega}}^{b_i}\|_{b_i}^2=-\frac{\partial^2}{\partial\omega\partial\overline{\omega}}\log\|k_\omega^{b_i}\|_{b_i}^2,
\]
for all $\omega\in\mathbb D$. But
\[
\|k_\omega^{b_i}\|_{b_i}^2=\frac{1-|b_i(\omega)|^2}{1-|\omega|^2},
\]
and then
\[
\log\|k_\omega^{b_i}\|_{b_i}^2=\log(1-|b_i(\omega)|^2)-\log(1-|\omega|^2),
\]
for all $\omega\in\mathbb D$. According to Lemma~\ref{lem:calcul}, we thus get
\[
\clk_{S_{b_i}^*}(\omega)=-\frac{|b_i'(\omega)|^2}{(1-|b_i(\omega)|^2)^2}+\frac{1}{(1+|\omega|^2)^2},
\]
for all $\omega\in\mathbb D$. Therefore $S^*_{b_1} \cong S^*_{b_2}$ is equivalent to
\[
-\frac{|b_1'(\omega)|^2}{(1-|b_1(\omega)|^2)^2}+\frac{1}{(1+|\omega|^2)^2}=-\frac{|b_2'(\omega)|^2}{(1-|b_2(\omega)|^2)^2}+\frac{1}{(1+|\omega|^2)^2},
\]
that is,
\[
\frac{|b_1'(\omega)|}{1-|b_1(\omega)|^2}=\frac{|b_2'(\omega)|}{1-|b_2(\omega)|^2},
\]
for all $\omega\in\mathbb D$. This completes the proof of the theorem.
\end{proof}

\section{Carath\'eodory angular derivatives}\label{sec: derivatives}

Recently, there has been particular interest in the case where $b$ is a rational (not inner) function in $H^\infty_1$. Indeed, in that case, we have a nice description of the space $\mathcal H(b)$. See for instance \cite{MR4653340, MR3503356,MR4404446,MR4596028}. In this particular case, we prove that condition \eqref{eq:ue-Sb} appearing in Theorem \ref{Thm-Sb-unitary-equivalent} is connected to Carath\'eodory angular derivatives. Recall that a function $b\in H^\infty_1$ has a Carath\'eodory angular derivative (an ADC for short) at a point $\zeta\in\mathbb T$ if $b$ and $b'$ have a non-tangential limit at $\zeta$, with 
\[
|b(\zeta)|=1.
\]
A well known result of Carath\'eodory says that $b$ has an ADC at $\zeta$ if and only if
\[
c=\liminf_{z\to\zeta}\frac{1-|b(z)|^2}{1-|z|^2}<\infty.
\]
Moreover, in that case, we have (cf. \cite[Theorem 21.1]{FM2})
\[
c>0,
\]
and 
\[
c=\zeta b'(\zeta)\overline{b(\zeta)}.
\]
The following theorem is essentially in the context of Theorem \ref{Thm-Sb-unitary-equivalent}. More specifically, we apply the identity given in \eqref{eq:ue-Sb}.

\begin{Theorem}\label{cor:ADC}
 Let $b_1,b_2 \in H^\infty_1$ be two rational (not inner) functions in $H^\infty_1$. Assume that
\[
S_{b_1} \cong S_{b_2}.
\]
Then $b_1$ and $b_2$ must have the same Carath\'eodory angular derivative points.
\end{Theorem}

\begin{proof}
It is known (see \cite[Lemma 3.1]{MR3503356}) that $b_1$ and $b_2$ are non extreme points in $H^\infty_1$. Then, according to Theorem~\ref{Thm-Sb-unitary-equivalent}, for every $\omega\in\mathbb D$, we have
\begin{equation}\label{eq2:ue-Sb}
\frac{|b_1'(\omega)|}{1-|b_1(\omega)|^2}=\frac{|b_2'(\omega)|}{1-|b_2(\omega)|^2}.
\end{equation}
Assume now that $\zeta\in\mathbb T$ is a point where $b_1$ has a Carath\'eodory angular derivative. This means that there exists a sequence $(\omega_n)_n$ in $\mathbb D$ such that $\omega_n\to\zeta$ as $n\to \infty$, and
\[
\lim_{n\to\infty}\frac{1-|b_1(\omega_n)|^2}{1-|\omega_n|^2}=c=|b_1'(\zeta)|>0.
\]
Using the fact that $b_1'$ is continuous on the closed unit disk, this implies that
\[
\frac{1-|\omega_n|^2}{1-|b_1(\omega_n)|^2}|b_1'(\omega_n)|^2\longrightarrow \frac{1}{c}|b_1'(\zeta)|=1,\qquad\mbox{as }n\to\infty.
\]
Hence, it follows from \eqref{eq2:ue-Sb} that
\[
\frac{1-|\omega_n|^2}{1-|b_2(\omega_n)|^2}|b_2'(\omega_n)|^2\longrightarrow 1,\qquad\mbox{as }n\to\infty.
\]
But, since $b_2$ is a rational function in $H^\infty_1$, the function $b_2'$ is also bounded, and then for $n$ sufficiently large, we have
\[
\frac{1}{2}\leq \frac{1-|\omega_n|^2}{1-|b_2(\omega_n)|^2} |b'_2(\omega_n)|\leq \frac{1-|\omega_n|^2}{1-|b_2(\omega_n)|^2} \|b_2'\|_\infty.
\]
In particular,
\[
\liminf_{n\to+\infty}\frac{1-|b_2(\omega_n)|^2}{1-|\omega_n|^2}<\infty,
\]
which gives that $b_2$ has a Carath\'eodory angular derivative at $\zeta$. The above argument is of course symmetric proving the result.
\end{proof}

In this way, we have connected Carath\'eodory angular derivatives with Cowen-Douglas class operators and their curvatures.

\section{An example}\label{sec: examples}

This section aims to illustrate the unitary equivalance results with even a more concrete example. Define $b_1, b_2 \in H^\infty_1$ by
\[
b_1(z)=\frac{1+z}{2} \text{ and } b_2(z)=\frac{1-z}{2},
\]
for $z\in\mathbb D$. It is easy to see that $b_1$ and $b_2$ are non extreme points. Moreover, $(b_1,b_2)$ and $(b_2,b_1)$ are pythagorean pairs. Following the notations of Theorem \ref{Thm:Xb-unitary-equivalence}, we have
\[
\Phi_1(z)=\frac{b_1(z)}{b_2(z)}=\frac{1+z}{1-z},
\]
and
\[
\Phi_2(z)=\frac{b_2(z)}{b_1(z)}=\frac{1-z}{1+z}.
\]
As $\Phi_1'(z)=\frac{2}{(1-z)^2}$, it follows that
\[
\frac{|\Phi_1'(z)|}{1+|\Phi_1(z)|^2}=\frac{2}{|1-z|^2+|1+z|^2}=\frac{1}{1+|z|^2},
\]
for all $z \in \D$. Similarly, $\Phi_2'(z)=-\frac{2}{(1+z)^2}$ implies
\[
\frac{|\Phi_2'(z)|}{1+|\Phi_2(z)|^2}=\frac{2}{|1-z|^2+|1+z|^2}=\frac{1}{1+|z|^2},
\]
for all $z \in \D$. Therefore, for every $z\in\mathbb D$,
\[
\frac{|\Phi_1'(z)|}{1+|\Phi_1(z)|^2}=\frac{|\Phi_2'(z)|}{1+|\Phi_2(z)|^2},
\]
and Theorem~\ref{Thm:Xb-unitary-equivalence} implies that $X_{b_1} \cong X_{b_2}$.

In this example, we also note that
\[
\mathcal H(b_1) \neq \mathcal H(b_2).
\]
Indeed, since $b_1$ and $b_2$ are two non extreme points of $H^\infty_1$, it follows from \cite{Ball-Kriete} (see also \cite[Corollary 27.17]{FM2}) that if $\mathcal H(b_1)=\mathcal H(b_2)$, then it would imply that
\[
a_1a_2^{-1}\in H^\infty.
\]
But $a_1=b_2$ and $a_2=b_1$, and thus
\[
\frac{a_1(z)}{a_2(z)}=\frac{b_2(z)}{b_1(z)}=\frac{1-z}{1+z}.
\]
Clearly,
\[
\frac{1-z}{1+z} \notin H^\infty.
\]
Therefore $\mathcal H(b_1)\neq \mathcal H(b_2)$. Note that in this example, Corollary~\ref{cor:ADC} implies that $S_{b_1}$ and $S_{b_2}$ are not unitarily equivalent. Indeed, in that case, it is easy to see that $b_1$ has a Carath\'eodory angular derivative only at point $1$, whereas $b_2$ has a Carath\'eodory angular derivative only at point $-1$.


It would be fascinating to explore how the results of this paper extend to several variables or more general domains. While many significant results are known in this direction, particularly in the context of Cowen-Douglas class operators \cite{Misra} and related theories in several variables \cite{Zheng}, further progress is needed. This is especially true given the sophistication required in the study of de Branges-Rovnyak spaces in several variables.

\bibliographystyle{plain}

\bibliography{references}

\end{document}